\documentclass[a4paper,10pt]{amsart}

\usepackage[english]{babel}
\usepackage[utf8]{inputenc}

\usepackage{amssymb}
\usepackage{amsmath}
\usepackage{amsthm}
\usepackage{bbm}

\usepackage{enumitem}
\usepackage{tikz}
\usepackage{marginnote}

\newcommand{\bbC}{\mathbb{C}}

\newcommand{\bbN}{\mathbb{N}}

\newcommand{\bbR}{\mathbb{R}}

\newcommand{\suchthat}{\,|\,} 
\DeclareMathOperator{\id}{id} 
\DeclareMathOperator{\re}{Re} 
\newcommand{\norm}[1]{\left\lVert #1 \right\rVert} 
\newcommand{\modulus}[1]{\left\lvert #1 \right\rvert} 
\DeclareMathOperator{\dom}{dom} 

\newcommand{\spec}{\sigma} 
\newcommand{\Res}{\mathcal{R}} 

\newcommand{\iu}{\mathrm{i}} 

\theoremstyle{definition}
\newtheorem{definition}{Definition}
\newtheorem{remark}[definition]{Remark}

\newtheorem{examples}[definition]{Examples}

\theoremstyle{plain}
\newtheorem{proposition}[definition]{Proposition}
\newtheorem{lemma}[definition]{Lemma}
\newtheorem{theorem}[definition]{Theorem}
\newtheorem{corollary}[definition]{Corollary}

\begin{document}

\title[Analyticity is inherited under domination]{Analyticity of positive semigroups is inherited under domination}
\author{Jochen Gl\"uck}
\address{Jochen Gl\"uck, University of Wuppertal, School of Mathematics and Natural Sciences, Gaußstr.\ 20, 42119 Wuppertal, Germany}
\email{glueck@uni-wuppertal.de}
\date{\today}
\dedicatory{Dedicated to the memory of Manfred P.\ H.\ Wolff}
\begin{abstract}
	For positive $C_0$-semigroups $S$ and $T$ on a Banach lattice such that $S(t) \le T(t)$ for all times $t$, we prove that analyticity of $T$ implies analyticity of $S$. This answers an open problem posed by Arendt in 2004.
	
	Our proof is based on a spectral theoretic argument: we apply spectral theory of positive operators to multiplication operators that are induced by $S$ and $T$ on a vector-valued function space.
\end{abstract}

\maketitle

In a problem posed in \cite{Nagel2004}, Arendt asked whether analyticity of positive semigroups is inherited under domination.
The purpose of this note is to give a positive answer to this question:

\begin{theorem}
	\label{thm:intro-dom}
	Let $S,T$ be $C_0$-semigroups on a complex Banach lattice $E$ such that $0 \le S(t) \le T(t)$ for all $t \in [0,\infty)$.
	If $T$ is analytic, then so is $S$.
\end{theorem}

The sector on which $T$ is analytic is not inherited by $S$: 
the angle of analyticity for $S$ can be strictly larger 
(Example~\ref{exas:sector}\ref{exas:sector:itm:larger-for-S}) 
or strictly smaller 
(Example~\ref{exas:sector}\ref{exas:sector:itm:smaller-for-S}) 
than for $T$.

Throughout, we freely use the theory of Banach lattices (see e.g.\ \cite{Schaefer1974}) and of $C_0$-semigroups (see e.g.\ \cite{EngelNagel2000}). 
Standard references for positive semigroups are \cite{ArendtEtAl1986, BatkaiKramarFijavzRhandi2017}, and domination between semigroups is discussed in detail in \cite[Section~C-II.4]{ArendtEtAl1986}.

We prove Theorem~\ref{thm:intro-dom} before Remark~\ref{rem:non-positive}.
The proof is surprisingly easy if one employs two appropriate results from the literature and reformulates one of them in a slick way.
We begin by stating those results in the following two propositions.

The first proposition is (a special case of) a result by Kato \cite{Kato1970}; 
a simplified version of the proof of ``(ii) $\Rightarrow$ (i)'' can be found in \cite[Lemma~2.1]{Fackler2013}.

\begin{proposition}[Kato]
	\label{prop:kato}
	For a $C_0$-semigroup $T$ on a complex Banach space $X$ the following assertions are equivalent:
	\begin{enumerate}[label=\upshape(\roman*)]
		\item 
		The semigroup $T$ is analytic.
		
		\item 
		There exists a time $t_0 > 0$ and a complex number $\lambda$ of modulus $1$ with the following property:
		there is a constant $K > 0$ such that
		\begin{align*}
			\lambda \not\in \spec(T(t)) 
			\qquad \text{and} \qquad 
			\norm{\Res\big(\lambda, T(t)\big)} \le K
		\end{align*}
		for all $t \in (0, t_0]$.
	\end{enumerate}
\end{proposition}

Here, $\spec(T(t))$ denotes the spectrum of $T(t)$ and $\Res\big(\lambda, T(t)\big) := \big(\lambda - T(t)\big)^{-1}$ the resolvent of $T(t)$ at $\lambda$.

The second proposition belongs to the spectral theory of positive operators 
(sometimes referred to as infinite-dimensional \emph{Perron--Frobenius theory}, e.g.\ in \cite{ArendtEtAl1986, Greiner1981}).
The result was proved, even under slightly more general conditions, by Räbiger and Wolff \cite[Theorem~1.4]{RaebigerWolff1997};
their proof employs techniques which were introduced earlier by Lotz \cite{Lotz1968}.

\begin{proposition}[Räbiger--Wolff]
	\label{prop:raebiger-wolff}
	Let $0 \le S \le T$ be bounded linear operators on a complex Banach lattice $E$ and assume that $T$ is power bounded, i.e.\ $\sup_{n \in \bbN_0} \norm{T^n} < \infty$.
	If $\lambda \in \spec(S)$ and $\modulus{\lambda} = 1$, then also $\lambda \in \spec(T)$.
\end{proposition}

In an attempt to prove Theorem~\ref{thm:intro-dom}, it seems natural to first try the following approach:

There is no loss of generality in assuming that the semigroup $T$ is bounded, i.e.\ that we have $\sup_{t \in [0,\infty)} \norm{T(t)} < \infty$.
Now let $T$ be analytic and let $t_0$ and $\lambda$ be as in Proposition~\ref{prop:kato}(ii). 
For each $t \in (0,t_0]$ we then have $\lambda \not\in \spec(T(t))$. 
Since, by the boundedness of $T$, $T(t)$ is power-bounded, and $0 \le S(t) \le T(t)$, it follows from Proposition~\ref{prop:raebiger-wolff} that $\lambda \not\in \spec(S(t))$.
So the only difficulty is to show that the resolvent $\Res(\lambda, S(t))$ is uniformly bounded as $t$ runs through $(0, t_0]$.

We may even assume that $\norm{T(t)} \to 0$ as $t \to \infty$, which means that the spectral radius of $T(t)$ is strictly less than $1$ for each $t > 0$.
Then the Neumann series representation of the resolvent immediately yields the estimate
\begin{align*}
	\modulus{\Res\big(\lambda,S(t)\big)x} \le \Res\big(\modulus{\lambda}, T(t)\big) \modulus{x} = \Res\big(1, T(t)\big) \modulus{x}
\end{align*}
for all $x \in E$. 
However, this does not give the desired boundedness since the spectral radius of $T(t)$ could be close to $1$ for small $t$, and thus $\Res(1, T(t))$ cannot be expected to be bounded as $t \downarrow 0$.

To solve this problem, we show now that property~(ii) in Kato's characterization can be rephrased as a spectral property of a single operator that acts on a vector-valued function space. 
This reformulation is based on the following simple lemma for general families of operators.

\begin{lemma}
	\label{lem:invertible-multiplication-op}
	Let $X$ be a (real or complex) Banach space, let $I$ be a non-empty set and let $T = (T_i)_{i \in I}$ be a norm bounded family of bounded linear operators on $X$. 
	Consider the operator
	\begin{align*}
		\hat T: \ell^\infty(I; X) \to \ell^\infty(I; X)
	\end{align*}
	given by
	\begin{align*}
		(\hat T f)(i) = T_i f(i)
	\end{align*}
	for each $f \in \ell^\infty(I; X)$ and $i \in I$. 
	
	Then $\hat T$ is bijective if and only if each of the operators $T_i$ is bijective and $\sup_{i \in I} \norm{T_i^{-1}} < \infty$.
\end{lemma}

This is certainly well-known among experts in operator theory and related results can be found in different places in the literature, for instance in \cite[Section~2]{Holderrieth1991} and \cite[Section~2]{Heymann2014}. 
We include the proof to demonstrate that it is particularly simple in the situation of Lemma~\ref{lem:invertible-multiplication-op}.

\begin{proof}[Proof of Lemma~\ref{lem:invertible-multiplication-op}]
	``$\Leftarrow$'' 
	This implication is straightforward to check.
	
	``$\Rightarrow$''
	Assume that $\hat T$ is bijective. For every $f \in \ell^\infty(I; X)$ and every $i \in I$ we then have
	\begin{align}
		\label{eq:lem:invertible-multiplication-op}
		f(i) = (\hat T  \,  \hat T^{-1} f)(i) = T_i \; (\hat T^{-1} f)(i).
	\end{align}
	By substituting functions for $f$ which are $0$ at each but one position, we can thus see that every operator $T_i$ is surjective.
	Similarly, we have $f = \hat T^{-1} \hat T f$ for each $f \in \ell^\infty(I; X)$, and by again substituting vectors for $f$ which are $0$ at each but one position, we can see that every operator $T_i$ is injective.
	
	Since we now know that each $T_i$ is bijective, equation~\eqref{eq:lem:invertible-multiplication-op} implies
	\begin{align*}
		T^{-1}\big(f(i)\big) = (\hat T^{-1} f)(i)
	\end{align*}
	for each $f \in \ell^\infty(I; X)$ and each $i \in I$. 
	We once again substitute vectors for $f$ that are $0$ at each but one position, and thus see that $\norm{T_i^{-1}} \le \norm{\hat T^{-1}}$ for each $i \in I$.
\end{proof}

As a direct consequence of the lemma we can reformulate Kato's characterization from Proposition~\ref{prop:kato}. 
We use the following notation:
for a $C_0$-semigroup $T$ on a Banach space $X$ and for $t_0 > 0$ we define a bounded linear operator
\begin{align*}
	\hat T_{t_0}: \ell^\infty\big((0,t_0]; X\big) \to \ell^\infty\big((0,t_0]; X\big)
\end{align*}
by 
\begin{align*}
	\big(\hat T_{t_0} f)(t) = T(t) f(t)
\end{align*}
for each $f \in \ell^\infty((0,t_0]; X)$ and each $t \in (0,t_0]$.
With this notation, we now obtain immediately:
	
\begin{corollary}
	\label{cor:kato-reformulated}
	For a $C_0$-semigroup $T$ on a complex Banach space $X$ the following assertions are equivalent:
	\begin{enumerate}[label=\upshape(\roman*)]
		\item 
		The semigroup $T$ is analytic.
		
		\item 
		There exists a time $t_0 > 0$ and a complex number $\lambda$ of modulus $1$ such that $\lambda \not\in \spec(\hat T_{t_0})$.
	\end{enumerate}
\end{corollary}

With this formulation of Kato's characterization, the proof of our main result is very easy:

\begin{proof}[Proof of Theorem~\ref{thm:intro-dom}]
	After rescaling $S$ and $T$ we may, and shall, assume that $T$ is bounded.
	Since $T$ is analytic, there exists, by Corollary~\ref{cor:kato-reformulated}, a time $t_0 > 0$ and a complex number $\lambda$ of modulus $1$ such that $\lambda \not\in \spec(\hat T_{t_0})$, where we use the notation $\hat T_{t_0}$ introduced before Corollary~\ref{cor:kato-reformulated}.
	
	We use the same notation for $S$ and thus have $0 \le \hat S_{t_0} \le \hat T_{t_0}$ 
	(where $\ell^\infty\big((0,t_0]; E\big)$ is ordered pointwise). 
	Since $\hat T_{t_0}$ is power bounded (due to the boundedness of the semigroup $T$), Proposition~\ref{prop:raebiger-wolff} shows that $\lambda \not\in \spec(\hat S_{t_0})$. 
	Hence, again by Corollary~\ref{cor:kato-reformulated}, the semigroup $S$ is analytic.
\end{proof}

\begin{remark}
	\label{rem:non-positive}
	Theorem~\ref{thm:intro-dom} does not remain true if we consider domination of non-positive semigroups $S$, i.e., if we only assume
	\begin{align}
		\label{eq:rem:non-positive}
		\modulus{S(t)x} \le T(t) \modulus{x}
	\end{align}
	for all $x \in E$ and $t \in [0,\infty)$ instead of $0 \le S(t) \le T(t)$. 
	
	As a simple counterexample, let $p \in [1,\infty)$ and consider the semigroup $S$ on $\ell^p$ given by
	\begin{align*}
		(S(t)f)(n) = e^{\iu tn} f(n)
	\end{align*}
	for all $f \in \ell^p$ and $n \in \bbN$. 
	This semigroup is clearly not analytic, but it satisfies the domination condition~\eqref{eq:rem:non-positive} for $T(t) = \id$. 
	
	However, it seems that this example cannot be directly adapted to obtain a counterexample over the real field.
	More generally speaking, the author does not know whether analyticity of a positive semigroup $T$ together with~\eqref{eq:rem:non-positive} implies analyticity of $S$ if $S$ is not positive but leaves the real part of the underlying Banach lattice invariant.
\end{remark}

We close the article with several counterexamples that show that the angle 
of bounded analyticity or analyticity is not inhertied by the smaller semigroup.

\begin{examples}
	\label{exas:bdd-on-sector}
	In the situation of Theorem~\ref{thm:intro-dom} it can happen, 
	even in finite dimensions, 
	that $S$ is bounded on a strictly smaller or on a strictly larger complex sector than $T$.
	
	\begin{enumerate}[label=(\alph*)]
		\item\label{exas:bdd-on-sector:itm:larger-for-S} 
		\emph{$S$ has a strictly larger angle of bounded analyticity than $T$:}
		Consider the $3 \times 3$ permutation matrix
		\begin{align*}
			P
			:= 
			\begin{pmatrix}
				0 & 1 & 0 \\ 
				0 & 0 & 1 \\ 
				1 & 0 & 0 
			\end{pmatrix}
		\end{align*}
		and set $A := -\id$ and $B := P - \id$.
		The semigroup $T$ generated by $B$ dominates the semigroup $S$ generated by $A$ and both are positive.
		Morever, $S(z) = e^{zA}$ is uniformly bounded for $\re z > 0$, 
		but $T(z) = e^{zB}$ is not since $B$ has the eigenvaue $e^{2\pi \iu/3} - 1$ which is not in $(-\infty,0]$.
		
		\item\label{exas:bdd-on-sector:itm:smaller-for-S} 
		\emph{$S$ has a strictly smaller angle of bounded analyticity than $T$:}
		Consider the $3 \times 3$ matrices
		\begin{align*}
			A_0 
			:= 
			\begin{pmatrix}
				0 & 1 & 0 \\ 
				0 & 0 & 1 \\ 
				1 & 0 & 0 
			\end{pmatrix}
			\quad \text{and} \quad 
			B_0 
			:= 
			\begin{pmatrix}
				0 & 1 & 2 \\ 
				0 & 0 & 1 \\ 
				1 & 0 & 0 
			\end{pmatrix}
			.
		\end{align*}
		The eigenvalues of $A_0$ are $1$, $e^{2\pi \iu/3}$, and $e^{4\pi \iu/3}$, 
		and a brief computation shows that $B_0$ has the eigenvalues 
		$-1$, $\frac{1 - \sqrt{5}}{2}$, and $\frac{1 + \sqrt{5}}{2} =: \varphi$,
		which are all real. 
		Now let $S$ and $T$ be the semigroups on $\bbC^3$ generated by the matrices 
		\begin{align*}
			A := A_0 - \varphi I 
			\qquad \text{and} \qquad 
			B := B_0 - \varphi I
			,
		\end{align*}
		respectively. 
		Both semigroups are positive and $T$ dominates $S$ since $B \ge A$. 
		As $B$ is diagonalizable and all its eigenvalues are located in $(-\infty,0]$ 
		we see that $T(z) = e^{zB}$ is uniformly bounded for $\re z > 0$. 
		This boundedness property is not true for $S(z) = e^{zA}$, though, since $A$ has the eigenvalue $e^{2\pi \iu/3} - \varphi$ 
		which is not in $(-\infty,0]$.
	\end{enumerate}
\end{examples}

The author does not know whether boundedness of $T$ on a sector implies boundedness of $S$ on a (in general, smaller) sector.

\begin{examples}
	\label{exas:sector}
	In the situation of Theorem~\ref{thm:intro-dom} it can happen 
	that the angle of analyticity for $S$ is strictly smaller or stricly larger than for $T$:
	
	\begin{enumerate}[label=(\alph*)]
		\item\label{exas:sector:itm:larger-for-S} 
		\emph{$S$ has a strictly larger angle of analyticity than $T$:}
		Let $A, B \in \bbR^{3 \times 3}$ be the matrices from 
		Example~\ref{exas:bdd-on-sector}\ref{exas:bdd-on-sector:itm:larger-for-S}
		and consider the Banach lattice $E := \ell^2(\bbN; \bbC^3)$ of all square summable sequences in $\bbC^3$. 
		Define a closed operator $\tilde B: E \supseteq \dom(\tilde B) \to E$ on this space by
		\begin{align*}
			\dom(\tilde B) & := \left\{x \in E \suchthat \big(n B x(n)\big)_{n \in \bbN} \in E \right\}, \\
			  \tilde B x    & := \big(n B x(n)\big)_{n \in \bbN},
		\end{align*}
		and likewise an operator $\tilde A$. 
		Then $\tilde A$ and $\tilde B$ generate positive $C_0$-semigroups $\tilde S$ and $\tilde T$ on $E$, 
		respectively,
		and $\tilde T$ dominates $\tilde S$.

		Observe that $\tilde S$ extends to an analytic (and bounded) semigroup on the entire right half plane
		since the matrices $e^{zA} \in \bbC^{3 \times 3}$ are uniformly bounded for $\re z > 0$. 
		However, for each $n \in \bbN$ the number $n(e^{2\pi\iu/3} - 1)$ is an eigenvalue of $\tilde B$, 
		so the maximal angle of analyticity of $\tilde T$ is strictly less than $\pi/2$.
	
		\item\label{exas:sector:itm:smaller-for-S}
		\emph{$S$ has a strictly smaller angle of analyticity than $T$:}
		This example can be constructed in the same way as in~\ref{exas:sector:itm:larger-for-S} 
		if one uses the matrices $A$ and $B$ from 
		Example~\ref{exas:bdd-on-sector}\ref{exas:bdd-on-sector:itm:smaller-for-S} instead.
	\end{enumerate}
\end{examples}

\subsection*{Acknowledgement} 

I am grateful to the referee for very useful corrections 
and for asking several questions that lead to the Examples~\ref{exas:bdd-on-sector} and~\ref{exas:sector}.

\bibliographystyle{plain}
\bibliography{literature}

\end{document}